\documentclass[a4paper,11pt]{amsart}

\usepackage[english]{babel}
\usepackage[utf8]{inputenc}

\usepackage{amssymb}
\usepackage{amsmath}
\usepackage{amsthm}
\usepackage{bbm}

\usepackage{enumitem}
\usepackage{tikz}
\usepackage{marginnote}

\usepackage[pagebackref,colorlinks=true,linkcolor=blue,citecolor=blue,urlcolor=red, hypertexnames=true]{hyperref}

\newcommand{\bbC}{\mathbb{C}}

\newcommand{\bbN}{\mathbb{N}}

\newcommand{\bbR}{\mathbb{R}}

\newcommand{\calL}{\mathcal{L}}

\newcommand{\suchthat}{\,|\,} 
\DeclareMathOperator{\id}{id} 
\newcommand{\norm}[1]{\left\lVert #1 \right\rVert} 

\newcommand{\spec}{\sigma} 
\newcommand{\sa}{{\operatorname{sa}}}

\theoremstyle{definition}
\newtheorem{definition}{Definition}[section]

\newtheorem{remarks}[definition]{Remarks}

\newtheorem{example}[definition]{Example}

\theoremstyle{plain}

\newtheorem{theorem}[definition]{Theorem}
\newtheorem{corollary}[definition]{Corollary}

\numberwithin{equation}{section}

\begin{document}

\title[The Huijsmans--de Pagter problem in finite dimensions]{A note on the Huijsmans--de Pagter problem on finite dimensional ordered vector spaces}
\author{Catalin Badea}
\address[C. Badea]{Univ Lille, CNRS, UMR 8524 - Laboratoire Paul Painlevé, France}
\email{cbadea@univ-lille.fr}
\author{Jochen Gl\"uck}
\address[J. Gl\"uck]{Bergische Universit\"at Wuppertal, Fakult\"at f\"ur Mathematik und Naturwissenschaften, Gaußstr.\ 20, 42119 Wuppertal, Germany}
\email{glueck@uni-wuppertal.de}
\date{4th of May, 2024}
\begin{abstract} 
	A classical problem posed in 1992 by Huijsmans and de Pagter asks whether, 
	for every positive operator $T$ on a Banach lattice with spectrum $\sigma(T) = \{1\}$,
	the inequality $T \ge \operatorname{\id}$ holds true. 
	While the problem remains unsolved in its entirety, a positive solution is known in finite dimensions.
	In the broader context of ordered Banach spaces, Drnovšek provided an infinite-dimensional counterexample. 
	In this note, we demonstrate the existence of finite-dimensional counterexamples, 
	specifically on the ice cream cone and on a polyhedral cone in $\mathbb{R}^3$. 
	On the other hand, taking inspiration from the notion of $m$-isometries, we establish that each counterexample must contain a Jordan block of size at least $3$.
\end{abstract}
 \thanks{This work was supported in part by the Labex CEMPI (grant ANR-11-LABX-0007-01) and by the French-German 2023 project ``Théorie des opérateurs et interactions'' of University of Lille.}

\maketitle

\section{Introduction} 
\label{section:introduction}
In 1941, Gelfand \cite{Gelfand1941} proved that a doubly power-bounded element $a$ of a complex unital Banach algebra with spectrum $\spec(a) = \{1\}$ is the identity element $1$. Here, double power-boundedness means that $a$ is invertible and $\sup_{n\in \mathbb{Z}}\|a^n\| < +\infty$. For a comprehensive survey of this result and of related topics, readers are directed to \cite{Zemanek1994}.

In \cite[Problem 11, Question 1, p.\,149]{HuijsmansLuxemburg1992}, Huijsmans and de Pagter posed the following question: If $T$ is a positive linear operator on a Banach lattice with spectrum $\spec(T) = \{1\}$, does it follow that $T \geq \operatorname{id}$? In this manuscript $\id$ denotes the identity operator.

The Huijsmans--de Pagter problem in its full generality is still open, although the answer is known to be positive in various special cases. For instance, if $T$ is disjointness preserving, then $\spec(T) = \{1\}$ implies $T = \operatorname{id}$ (even if $T$ is not assumed to be positive), as shown in \cite{Arendt1983, HuijsmansDePagter1992}. This includes the special case where $T$ is a lattice homomorphism, addressed earlier in \cite{SchaeferWolffArendt1978, Huijsmans1988, Huijsmans1989}. For positive operators whose negative powers satisfy a certain spectral growth condition, a positive answer is provided in \cite[Theorem 5.1]{Zhang1993}. This encompasses the case of positive operators whose single spectral value $1$ is a pole of the resolvent, as a special case \cite[Theorem 5.3]{Zhang1993}.

Particularly, the answer is positive in finite dimensions, with a much simpler proof available in this case, as seen in \cite[Theorem 4.1]{Zhang1993}; see also \cite[Theorems~ 4.1 and ~4.3]{Mouton2003}. For a further discussion of the problem in infinite dimensions, we also refer to \cite{Zhang1992}.

It is natural to inquire about the same problem when the positive operator $T$ does not act on a Banach lattice but, more generally, on an ordered Banach space.
In \cite[Theorem 2]{Drnovsek2018}, Drnovšek constructed an infinite-dimensional counterexample. Specifically, he constructed an ordered Banach algebra $\mathcal{A}$ and a positive element $a \in \mathcal{A}$ such that $\spec(a) = {1}$, but $a \not\geq 1$, where $1$ denotes the identity element in $\mathcal{A}$. Consequently, the operator $T: \mathcal{A} \to \mathcal{A}$, $b \mapsto ab$, is also positive and satisfies $\spec(T) = {1}$, but $T \not\geq \operatorname{id}$ since $T1 = a \not\geq 1 = \operatorname{id} 1$. Drnovšek observed in \cite[Proposition 3]{Drnovsek2018} that his construction cannot be adapted to the finite-dimensional case. 
Related results in ordered Banach algebras can be found in \cite{BraatvedtBritsRaubenheimer2009}.

In this note, we provide two finite-dimensional counterexamples (Examples~\ref{exa:ice-cream} and~\ref{exa:four-ray}) that give a negative answer to the Huijsmans--de Pagter problem for positive linear operators on finite-dimensional ordered vector spaces. In the first example, the cone corresponds to the ice cream cone in four (or three) dimensions. In the second example, the cone is polyhedral, with four extreme rays, in three dimensions.

In Section~\ref{section:generalized-ev}, we demonstrate a partial positive answer: if all Jordan blocks of $T$ have size $\leq 2$, then $\spec(T) = \{1\}$ and $T \geq 0$ implies that $T \geq \operatorname{id}$ (Corollary~\ref{cor:jordan-2}). However, the situation becomes more complex if $T$ has larger Jordan blocks; see Theorem~\ref{thm:generalized-ev-higher-rank} for details. Our arguments closely relate to the theory of $m$-isometries, see Remark~\ref{rem:m-isometries}(b).

\subsection*{Notation and terminology for ordered vector spaces}

An \emph{ordered vector space} refers to a real vector space $V$ together with a subset $V_+$ satisfying $\alpha V_+ + \beta V_+ \subseteq V_+$ for all scalars $\alpha, \beta \in [0,\infty)$, and $V_+ \cap -V_+ = \{0\}$. The set $V_+$ is referred to as the \emph{positive cone} or simply the \emph{cone} in the ordered vector space $V$. It induces a partial order $\leq$ on $V$ defined by $v \leq w$ if and only if $w-v \geq 0$. Consequently, the elements of $V_+$ are precisely those vectors $v$ satisfying $v \geq 0$, and they are termed the \emph{positive} vectors in $V$. A linear map $T: V \to V$ is called \emph{positive} if $TV_+ \subseteq V_+$. 
We write $T \ge 0$ to say that $T$ is positive. 
For two linear maps $S,T : V \to V$ we write $S \le T$ if $T-S$ is positive.

\section{Counterexamples}
\label{section:counterexamples}

For two matrices $L,R \in \bbC^{n \times n}$, consider the linear map
\begin{align*}
	T :	\bbC^{n \times n} \ni A \mapsto LAR \in \bbC^{n \times n}.
\end{align*}
One has $\spec(T) = \spec(L) \spec(R) := \{\lambda \mu \suchthat \lambda \in \spec(L), \, \mu \in \spec(R)\}$. 
Indeed, the inclusion $\supseteq$ easily follows by applying $T$
to the rank-$1$ matrices $yx^{\operatorname{T}}$, where $y \in \bbC^n$ is an eigenvector of $L$
and $x \in \bbC^n$ is an eigenvector of $R^T$.
For a proof of the converse inclusion $\subseteq$, we refer, for instance,
to \cite[Proposition~3.3.45, p.\,283]{HinrichsenPritchard2005}.
Alternatively, one can see the converse inclusion as follows: 
if $\spec(L) \spec(R)$ has $n^2$ distinct elements, the equality follows from cardinality reasons; 
the other case can be reduced to this case by a perturbation argument. 
We note that the same result even holds for bounded linear operators on infinite-dimensional Banach spaces, 
see~\cite[Theorem~10]{LumerRosenblum1959}.

\begin{example}[A counterexample on the Loewner/ice cream cone]
	\label{exa:ice-cream}
	Let $(\bbC^{2 \times 2})_\sa$ denote the four dimensional real vector space of all self-adjoint complex $2 \times 2$-matrices. 
	We endow it with the so-called \emph{Loewner cone} that consists of all positive semi-definite matrices.
	It is not difficult to check that this space is, as an ordered vector space, isomorphic to $\bbR^4$ 
	endowed with the ice cream cone
	\begin{align*}
		\{x \in \bbR^4 \suchthat x_1 \ge (x_2^2 + x_3^2 + x_4^2)^{1/2}\}.
	\end{align*}
	The complexification of $(\bbC^{2 \times 2})_\sa$ is the space $\bbC^{2 \times 2}$. 
	For the Jordan block 
	\begin{align*}
		J 
		:= 
		\begin{pmatrix}
			1 & 1 \\ 
			0 & 1
		\end{pmatrix}
	\end{align*}
	we define the linear map $T: \bbC^{2 \times 2} \to \bbC^{2 \times 2}$ by $TA := J^*AJ$. 
	This map leaves $(\bbC^{2 \times 2})_\sa$ invariant and defines a positive operator there. 
	As explained before the example, $T$ has spectrum $\{1\}$. 
	However, if we apply $T$ to the positive vector $\id_{\bbC^2} \in (\bbC^{2 \times 2})_\sa$, 
	we see that
	\begin{align*}
		T \id_{\bbC^2} 
		= 
		\begin{pmatrix}
			1 & 1 \\ 
			1 & 2
		\end{pmatrix}
		\not\ge 
		\id_{\bbC^2}.
	\end{align*}
	Indeed, $T\id_{\bbC^2}-\id_{\bbC^2}$ has determinant $-1$, and is thus not positive semi-definite.
	Therefore, $T \not\ge \id_{(\bbC^{2 \times 2})_\sa}$.
	
	We observe that $\norm{T^n}$ grows as $n^3$ as $n \to \infty$, 
	so $T$ has a generalized eigenvector of rank $3$ for the eigenvalue $1$.
	Furthermore, we observe that $T$ preserves the three-dimensional subspace of $(\bbC^{2 \times 2})_\sa)$ 
	consisting of all real symmetric matrices. 
	This space, endowed with the same order, is isomorphic to $\bbR^3$ with the ice cream cone. 
	Hence, there exists a counterexample even on $\mathbb{R}^3$ with the ice cream cone.	
\end{example}

\begin{example}[A counterexample on a four ray cone]
	\label{exa:four-ray}
	Consider the three canonical unit vectors $e_1, e_2, e_3 \in \bbR^3$ and set $z = (1,-1,1)^{\operatorname{T}}$. 
	Endow $\bbR^3$ with the cone $C$ spanned by $\{e_1,e_2,e_3,z\}$ 
	and consider the linear map $T: \bbR^3 \to \bbR^3$ whose representation matrix with respect to the canonical basis $(e_1, e_2, e_3)$ 
	is the Jordan block
	\begin{align*}
		\begin{pmatrix}
			1 & 1 & 0 \\ 
			0 & 1 & 1 \\ 
			0 & 0 & 1
		\end{pmatrix}
		.
	\end{align*}
	Clearly, $\spec(T) = \{1\}$. 
	We observe that $T$ is positive: indeed, $Te_k$ lies within the positive span of ${e_1,e_2,e_3}$, and thus, specifically in $C$, for each $k \in {1,2,3}$. 
	Moreover, $Tz = e_3 \in C$. However, one can easily verify that the vector 
	$(T-\id)z = (-1,1,0)^{\operatorname{T}}$ is not in $C$, 
	so $Tz \not\ge z$. Consequently, $T \not\ge \id$.
\end{example}

We note that a cone in $\bbR^3$ that is spanned by three rays is automatically a lattice cone 
\cite[Theorem~3.21]{AliprantisTourky2007}
and hence the Huijsmans--de Pagter problem has a positive answer for such a cone. 
Thus, the number of rays in Example~\ref{exa:four-ray} is minimal.

\section{Generalized eigenvectors}
\label{section:generalized-ev}

For the following results, we do not require the underlying vector space to be finite-dimensional.
Recall that a partially ordered vector space $V$ is termed \emph{Archimedean} if the following condition holds true for all $v,w \in V$: 
whenever $v \le \frac{1}{n}w$ for all integers $n \ge 1$, then $v \le 0$.
In finite dimensions, this condition is equivalent to the cone being closed.

\begin{theorem}
	\label{thm:generalized-ev-rank-2}
	Let $V$ be an Archimedean ordered vector space and let $T: V \to V$ be a positive linear map. 
	If $0 \le v \in V$ satisfies $(T-\id)^2 v \le 0$, then $Tv \ge v$.
\end{theorem}

Note that if $0 \leq v \in V$ is a generalized eigenvector of $T$ of rank $2$, then $(T-\operatorname{id})^2v = 0$, and thus the theorem implies that $Tv \geq v$. The proof idea is adapted from \cite[Lemma 1]{Richter1988}, where a similar argument was employed in the context of $2$-isometries and concave operators; please refer to Remark~\ref{rem:m-isometries}(b) for a more detailed explanation.

\begin{proof}[Proof of~Theorem~\ref{thm:generalized-ev-rank-2}]
	It follows from $(T-\id)^2 v \le 0$ that $T(T-\id)v \le (T-\id)v$. 
	By applying the positive operator $T^k$ to this inequality we get 
	$T^{k+1}(T-\id)v \le T^k(T-\id)v$ for all integers $k \ge 0$ and hence, 
	$T^{k}(T-\id)v \le (T-\id)v$ for all such $k$. 
	By summing up this inequality from $k=0$ to, say, $k=n-1\ge 0$ we obtain
	\begin{align*}
		T^n v - v \le n(T-\id)v,
	\end{align*}
	for each integer $n \ge 1$. Therefore $-\frac{1}{n}v \le (T-\id)v$ since $T^n v \ge 0$.
	As $V$ is Archimedean we conclude that $0 \le (T-\id)v$.
\end{proof}

The following corollary demonstrates that the operator $T$ in the counterexamples from the previous section must contain a Jordan block of size (at least) $3$.

\begin{corollary}
	\label{cor:jordan-2}
	Endow $\bbR^n$ with a closed cone. 
	If $T: \bbR^n \to \bbR^n$ is positive, satisfies $\spec(T) = \{1\}$ 
	and all its Jordan blocks have size $\le 2$, then $T \ge \id$.
\end{corollary}

\begin{proof}
	The assumptions on the spectrum and the Jordan blocks imply that $(T-\id)^2 = 0$. 
	In particular, for each $0 \le v \in \bbR^n$ we have $(T-\id)^2v = 0$ 
	and thus, by Theorem~\ref{thm:generalized-ev-rank-2}, $(T-\id)v \ge 0$. 
	Thus $T \ge \id$.
\end{proof}

\begin{remarks}
	\label{rem:m-isometries}
	(a) 
	The examples in Section~\ref{section:counterexamples} show 
	that $(\id-T)^3v = 0$ for a positive vector $v$ does not imply $Tv \ge v$.
	
	(b)
 	There is a close connection between our arguments and the theory of \emph{$m$-isometries}, 
 	which was introduced and studied in a series of papers by Agler and Stankus \cite{AglerStankus1995a, AglerStankus1995b, AglerStankus1996}.
 	Let $H$ be a complex Hilbert space, $\calL(H)$ denote the space of bounded linear operators on $H$, 
 	and let $L \in \calL(H)$. 
 	Similar to Example~\ref{exa:ice-cream}, we define a bounded linear operator $T: \calL(H) \to \calL(H)$ by  
 	\begin{align*}
 		TA = L^*AL
 	\end{align*}
 	for all $A \in \calL(H)$. 
 	For an integer $m \ge 1$, the operator $L$ is called an \emph{$m$-isometry} if $(\id_{\calL(H)} - T)^m \id_H  = 0$, 
 	where $\id_{\calL(H)} \in \calL\big(\calL(H)\big)$ denotes the identity operator on the Banach space $\calL(H)$ 
 	and $\id_H \in \calL(H)$ denotes the identity operator on the Hilbert space $H$. 
 	In other words, $L$ is an $m$-isometry if and only if $1$ is an eigenvalue of $T$ 
 	and $\id_H$ is a generalized eigenvector of rank no more than $m$ for this eigenvalue. 
 	Note that $L$ is a $1$-isometry if and only if it as an isometry.
 	
 	If we restrict $T$ to the self-adjoint part $V$ of $\calL(H)$ (which is an ordered vector space, 
 	with the positive cone given by the positive semi-definite operators), 
 	then $T$ is a positive operator. Applying Theorem~\ref{thm:generalized-ev-rank-2} to the vector $v = \id_H$ 
 	implies in particular that if $L$ is a $2$-isometry, then $L^*L \ge \id_H$, 
 	which means that $\norm{Lx} \ge \norm{x}$ for all $x \in H$. 
 	This result was previously established in \cite[Lemma 1]{Richter1988} (even for the so-called concave operators), 
	and the proof of Theorem~\ref{thm:generalized-ev-rank-2} above is an adaptation of the proof from \cite[Lemma 1]{Richter1988}.
\end{remarks}

Here is a more general version of Theorem~\ref{thm:generalized-ev-rank-2}. The reader is invited to compare 
the statement of the following result with~\cite[Theorem 2.5]{Gu2015}.

\begin{theorem}
	\label{thm:generalized-ev-higher-rank} 
	Let $V$ be an Archimedean ordered vector space and let $T: V \to V$ be a positive linear map. 
	If $0 \le v \in V$ satisfies $(T-\id)^{r+1} v \le 0$ for an integer $r \ge 0$, then $(T-\id)^r v \ge 0$.
\end{theorem}

For the proof we need the following simple observation about Archimedean ordered vector spaces: 
Let $V$ be an Archimedean ordered vector space and let $w_1, \dots, w_m \in V$. 
Moreover, let $(s^{(1)}_n)_{n \in \bbN}$, \dots, $(s^{(m)}_n)_{n \in \bbN}$ be sequences in $\bbR$ 
that converge to real numbers $s^{(1)}$, \dots, $s^{(m)}$, respectively. 
If 
\begin{align*}
	0 \le s^{(1)}_n w_1 + \dots + s^{(m)}_n w_m
\end{align*}
for all $n \in \bbN$, then $0 \le s^{(1)} w_1 + \dots + s^{(m)} w_m$.
A proof of this fact can, for instance, be found in \cite[Theorem~1.12]{AliprantisTourky2007}.

It is most instructive to give the proof of Theorem~\ref{thm:generalized-ev-higher-rank} first in the special case 
where one even assumes that $(T-\id)^{r+1} v = 0$. 
This special case is used in Corollary~\ref{cor:generalized-ev-positive} below and the proof in this case 
is quite simple and is an adaptation of \cite[Proposition~1.5]{AglerStankus1995a}:

\begin{proof}[Proof of Theorem~\ref{thm:generalized-ev-higher-rank} in the special case $(T-\id)^{r+1} v = 0$]
	For every $n \ge r$ one has 
	\begin{align*}
		0 
		\le 
		T^n v 
		= 
		\sum_{t=0}^n \binom{n}{t} (T-\id)^t v
		= 
		\sum_{t=0}^r \binom{n}{t} (T-\id)^t v
	\end{align*}
	since $(T-\id)^t v = 0$ for all integers $t \ge r+1$. 
	For fixed $t$, the binomial coefficient $\binom{n}{t}$ is a polynomial of degree $t$ in $n$, 
	so if we divide the above inequality by $\binom{n}{r}$ and let $n \to \infty$, 
	we obtain $0 \le (T-\id)^r v$ due to the consequence of the Archimedean property 
	that was mentioned right after Theorem~\ref{thm:generalized-ev-higher-rank}.
\end{proof}

For the more general case one cannot proceed so easily since $(T-\id)^{r+1} v \le 0$ 
does not imply $(T-\id)^t v \le 0$ for all integers $t \ge r+1$. The corresponding problem for $m$-isometries 
and related operators has been considered for instance in \cite{Gu2015}. The proof of \cite[Theorem 2.5]{Gu2015} does
not seem to adapt immediately to our situation. 
Instead, one can argue as follows:

\begin{proof}[Proof of Theorem~\ref{thm:generalized-ev-higher-rank} in the general case]
	It follows from $(T-\id)^{r+1} v \le 0$ that $T (T-\id)^r v \le (T-\id)^r v$.	
	Let us prove by induction over $s$ that, for all $s \in \{0, \dots, r\}$ and all $n \ge 0$, one has%
	\footnote{
		Recall that, for integers $n,k \ge 0$, the binomial coefficient $\binom{n}{k}$ is $0$ if $k > n$.
	} 
	\begin{align}
		\label{eq:thm:generalized-ev-higher-rank:inequality-s}
		T^n (T-\id)^{r-s} v \le \sum_{t=0}^s \binom{n}{t} (T-\id)^t (T-\id)^{r-s} v
	\end{align}
	For $s=0$ this follows for each $n \ge 0$ by applying the powers $T^0, \dots, T^{n-1}$ to the inequality 
	$T (T-\id)^r v \le (T-\id)^r v$, which indeed gives $T^n (T-\id)^r v \le (T-\id)^r v$.
	
	Now assume that~\eqref{eq:thm:generalized-ev-higher-rank:inequality-s} has already been proved for a number 
	$s \in \{0, \dots, r-1\}$ and all $n \ge 0$.
	By summing up this inequality for all $n$ from $0$ to, say, a number $n_1-1 \ge -1$, we obtain%
	\footnote{
		We employ the convention that a sum from $0$ to $-1$ is to be understood as the empty sum, which is defined as $0$.
	}
	\begin{align*}
		T^{n_1}(T-\id)^{r-s-1}v - (T-\id)^{r-s-1}v 
		& = 
		\sum_{n=0}^{n_1-1} T^n (T-\id)^{r-s} v
		\\
		& \le 
		\sum_{n=0}^{n_1-1} \sum_{t=0}^s \binom{n}{t} (T-\id)^t (T-\id)^{r-s} v 
		\\
		& = 
		\sum_{t=0}^s \binom{n_1}{t+1} (T-\id)^t (T-\id)^{r-s} v 
		\\
		& = 
		\sum_{t=1}^{s+1} \binom{n_1}{t} (T-\id)^t (T-\id)^{r-s-1} v
		,
	\end{align*}
	where we used the binomial identity $\sum_{n=0}^{n_1-1} \binom{n}{t} = \binom{n_1}{t+1}$, 
	which is true%
	\footnote{
	This binomial identity has the following nice combinatorial interpretation: when selecting 
	$t+1$ objects from $n_1$ objects labeled from $0$ through $n_1-1$, there are $\binom{n}{t}$ ways to accomplish this when the highest labeled object chosen is $n$.
	} 
	for all integers $n_1, t \ge 0$
	\cite[Formula~(11), p.\,56]{knuth}. 
	Thus,
	\begin{align*}
		T^{n_1}(T-\id)^{r-(s+1)}v 
		\le 
		\sum_{t=0}^{s+1} \binom{n_1}{t} (T-\id)^t (T-\id)^{r-(s+1)} v
		.
	\end{align*}
	This concludes the induction step and hence the proof of~\eqref{eq:thm:generalized-ev-higher-rank:inequality-s}.
	
	Now we use~\eqref{eq:thm:generalized-ev-higher-rank:inequality-s} for $s=r$. 
	Consequently, due to the positivity of $v$ and $T$, 
	\begin{align*}
		0 
		\le 
		T^n v 
		\le 
		\sum_{t=0}^r \binom{n}{t} (T-\id)^t v
	\end{align*}
	for each integer $n \ge 0$. 
	Now we divide by $\binom{n}{r}$ and let $n \to \infty$, 
	which gives $0 \le (T-\id)^r v$ due to the property of Archimedean spaces that was cited 
	right after Theorem~\ref{thm:generalized-ev-higher-rank}.
\end{proof}

Coming back to Theorem~\ref{thm:generalized-ev-higher-rank} in the special case $(T-\id)^{r+1} v = 0$, 
we note the following spectral theoretic corollary.

\begin{corollary}
	\label{cor:generalized-ev-positive}
	Let $V$ be an Archimedean ordered vector space and let $T: V \to V$ be a positive linear map. 
	If $\lambda \in (0,\infty)$ is an eigenvalue of $T$ with a generalized eigenvector in $V_+$, 
	then the eigenvalue $\lambda$ even has an eigenvector in $V_+$.
\end{corollary}

\begin{proof}
	By replacing $T$ with $T/\lambda$ we may assume that $\lambda = 1$. 
	Let $0 \not= v \in V_+$ be a generalized eigenvector of $T$ for the eigenvalue $1$. 
	If $v$ is itself an eigenvector, there is nothing to prove. 
	If $v$ is not an eigenvector, there exists an integer $r \ge 1$ such that $(T-\id)^{r+1} v = 0$, 
	but $w := (T-\id)^r v \not= 0$. 
	Hence, $w$ is an eigenvector of $T$ for the eigenvalue $1$ and 
	according to Theorem~\ref{thm:generalized-ev-higher-rank} one has $w \ge 0$.
\end{proof}

Recall from the introduction that a lattice isomorphism $T$ on a Banach lattice with spectrum $\{1\}$ 
is known to be the identity operator. 
It is natural to ask whether the same is true on more general ordered vector spaces -- 
where the natural replacement of the assumption that $T$ be a lattice isomorphism is that $T$ and $T^{-1}$ be both positive. 
Example~\ref{exa:ice-cream} shows that the answer is negative. 
However, one still has the following partial positive result. 

\begin{corollary}
	\label{cor:odd}
	Let $V$ be an Archimedean ordered vector space and let $T: V \to V$ be a bijective linear map 
	such that both $T$ and $T^{-1}$ are positive. 
	\begin{enumerate}[label=\upshape(\alph*)]
		\item 
		Let $r \ge 1$ be an odd integer and let $v \in V_+$. 
		If $(T-\id)^{r+1} v \le 0$, then $(T-\id)^r v = 0$.
	\end{enumerate}
	Assume now in addition that $V_+-V_+ = V$.
	\begin{enumerate}[resume, label=\upshape(\alph*)]
		\item 
		If $r \ge 1$ is an odd integer and $(T-\id)^{r+1} \le 0$, then $(T-\id)^r = 0$.
		
		\item 
		In particular, if $(T-\id)^2 \le 0$, then $T=\id$.
	\end{enumerate}
\end{corollary}

\begin{proof}
	(a) 
	By Theorem~\ref{thm:generalized-ev-higher-rank} we have $(T-\id)^r v \ge 0$.
	To also obtain the converse inequality, 
	first multiply the inequality $(T-\id)^{r+1} v \le 0$ with the positive operator $(T^{-1})^{r+1}$ 
	to obtain $(\id-T^{-1})^{r+1} v \le 0$. 
	Since $r+1$ is even, this implies that $(T^{-1} - \id)^{r+1} v \le 0$.
	Hence, another application of Theorem~\ref{thm:generalized-ev-higher-rank} gives $(T^{-1}-\id)^r v \ge 0$. 
	Now we multiply this inequality with the positive operator $T^r$ to obtain $(\id-T)^r v \ge 0$. 
	Since $r$ is odd, this gives $(T-\id)^r v \le 0$.
	
	(b) 
	This is an immediate consequence of~(a).
	
	(c) 
	This is a special case of~(b) for $r=1$.
\end{proof}

\bibliographystyle{plain}
\bibliography{literature}

\end{document}